\pdfoutput=1
\documentclass[3p,12pt]{elsarticle}

\usepackage{hyperref}
\usepackage{mathtools,amssymb,bm}
\usepackage{mathrsfs}
\usepackage{cases}
\usepackage{microtype}
\usepackage{cleveref}

\makeatletter
\def\ps@pprintTitle{
	\let\@oddhead\@empty
	\let\@evenhead\@empty
	\def\@oddfoot{\centerline{\thepage}}
	\let\@evenfoot\@oddfoot
}
\@addtoreset{equation}{section}
\@addtoreset{table}{section}
\@addtoreset{figure}{section}
\makeatother

\usepackage{graphicx,caption,subcaption,multirow,booktabs}

\usepackage{amsthm}
\newdefinition{mdef}{Definition}[section]
\newtheorem{prop}[mdef]{Proposition}
\newtheorem{lem}[mdef]{Lemma}
\newtheorem{thm}[mdef]{Theorem}
\newtheorem{cor}[mdef]{Corollary}
\newdefinition{fact}[mdef]{Fact}
\newdefinition{rmk}[mdef]{Remark}
\newdefinition{eg}[mdef]{Example}
\newdefinition{op}[mdef]{Open Problem}

\biboptions{numbers,sort&compress}

\begin{document}

\begin{frontmatter}

\title{Characterizations of annihilator $(b,c)$-inverses in arbitrary rings}

\author[1]{Chong-Quan~Zhang}
\ead{cqzhang@shu.edu.cn}
\author[1]{Qing-Wen~Wang\corref{cor}}
\ead{wqw@t.shu.edu.cn}
\author[1,2]{Huihui~Zhu}
\ead{hhzhu@hfut.edu.cn}
\address[1]{Department of Mathematics, Shanghai University, Shanghai 200444, P.R.China}
\address[2]{School of Mathematics, Hefei University of Technology, Hefei 230601, P.R.China}
\cortext[cor]{Corresponding author.}

\begin{abstract}
In this paper, we investigate some properties of annihilator $(b,c)$-inverses in an arbitrary ring. We demonstrate that one-sided annihilator $(b,c)$-inverses of elements in arbitrary rings may behave differently in contrast to one-sided $(b,c)$-inverses. Also, we discuss intertwining property, absorption law, reverse order law, and Cline's formula for annihilator $(b,c)$-inverses. As applications, we improve and extend some known results to $(b,c)$-inverses. In particular, we derive an equivalent condition of intertwining property for $(b,c)$-inverses in semigroups.
\end{abstract}


\end{frontmatter}

\section{Introduction}\label{sec 1}

Throughout this paper, $S$ will be denoted a semigroup, and $R$ will be used to denote an arbitrary ring (not necessarily unital). $S^1$ denotes the monoid equal to $S\cup\{1\}$ if $S$ has no identity element, and to $S$ otherwise. Since $R$ is also a semigroup under multiplication, $R^1$ denotes the monoid generated by $R$. For any $a\in R$, the left annihilator of $a$ is defined, as usual, by $\{x\in R:xa=0\}$, denoted by $\prescript{\circ}{}{a}$. Similarly, we denote by $a^\circ:=\{x\in R:ax=0\}$ the right annihilator of $a$. We also use the following notations.
\[
\prescript{\circ}{}{R}:=\{x\in R:xr=0,\forall r\in R\},\quad
R^\circ:=\{x\in R:rx=0,\forall r\in R\}.
\]

In 2011, Mary introduced a new type of generalized inverses in semigroups, called the inverse along an element, by Green's relations.
\begin{mdef}[{\cite[Definition 4]{mary_generalized_2011}}]\label{def_mary}
	Let $a,d\in S$. We call $x\in S$ an inverse of $a$ along $d$ if $x$ satisfies
	\begin{equation}\label{eq_def_mary}
	\mathrm{1)}~xad=d,\quad
	\mathrm{2)}~dax=d,\quad
	\mathrm{3)}~xS^1\subseteq dS^1,\quad
	\mathrm{4)}~S^1x\subseteq S^1d.
	\end{equation}
\end{mdef}
About one year afterward, Drazin independently defined a new class of outer generalized inverses in semigroups, called $(b,c)$-inverses, which is similar to the inverse along an element. Indeed, $x$ is an inverse of $a$ along $d$ if and only if $x$ is a $(d,d)$-inverse of $a$ by definitions.
\begin{mdef}[{\cite[Definition 1.3]{drazin_class_2012}}]\label{def}
	Let $a,b,c\in S$. We call $x\in S$ a $(b,c)$-inverse of $a$ if $x$ satisfies
	\begin{equation}\label{eq_def}
	\mathrm{1)}~xab=b,\quad
	\mathrm{2)}~cax=c,\quad
	\mathrm{3)}~x\in bSx,\quad
	\mathrm{4)}~x\in xSc.
	\end{equation}
\end{mdef}
For convenience, we will use the conditions below, which are equivalent to \eqref{eq_def} (as in \cite[Theorem 2.1]{drazin_left_2016}).
\begin{equation}\label{eq_def_useful}
\mathrm{1)}~xab=b,\quad
\mathrm{2)}~cax=c,\quad
\mathrm{3)}~x\in bS,\quad
\mathrm{4)}~x\in Sc.
\end{equation}
Actually, such $x$ is always unique (whenever it exists), denoted by $a^{(b,c)}$, and satisfies $xax=x$ ($x$ in this case is known as an outer generalized inverse of $a$). In addition, we call $y\in S$ a left (resp. right) $(b,c)$-inverse of $a$ if $y$ is a solution satisfying (\ref{eq_def_useful}.1) and (\ref{eq_def_useful}.4) (resp. (\ref{eq_def_useful}.2) and (\ref{eq_def_useful}.3)). Then, Drazin showed that several classical generalized inverses can be seen as special cases of $(b,c)$-inverses (see e.g. Lemma \ref{lem_relation} below). For more definitions and properties of $(b,c)$-inverses, we refer readers to the recent papers \cite{drazin_generalized_2014,zhu_further_2016,zhu_further_2018,mosic_bc_2018,ke_new_2018} and the reference therein.

Let $R$ be a $^*$-ring, namely, a ring with an antiautomorphism $^*$ that is its own inverse. We now recall some facts on generalized inverses. The Moore--Penrose inverse of $a\in R$, denoted by $a^\dagger$, is the unique solution to the following equations.
\begin{equation}\label{eq_def_mp}
\mathrm{1)}~axa=a,\quad
\mathrm{2)}~xax=x,\quad
\mathrm{3)}~(ax)^*=ax,\quad
\mathrm{4)}~(xa)^*=xa.
\end{equation}
The Drazin inverse of $a\in S$, denoted by $a^\mathrm{D}$, is the unique solution to the following equations.
\begin{equation}\label{eq_def_drazin}
\mathrm{1)}~a^{m+1}x=a^m\text{ for some positive integer }m,\quad
\mathrm{2)}~x^2a=x,\quad
\mathrm{3)}~ax=xa.
\end{equation}
The such least $m$ is called the Drazin index of $a$, denoted by $\operatorname{ind}(a)$.
\begin{lem}[{{\cite[p. 1910]{drazin_class_2012}}, or \cite[Theorem 11]{mary_generalized_2011}}]\label{lem_relation}
	The following statements hold.
	\begin{enumerate}[i)]
		\item Let $R$ be a $^*$-ring, and let $a\in R$. Then $a$ is Moore--Penrose invertible if and only if $a$ is $(a^*,a^*)$-invertible. In this case, $a^\dagger$ coincides with $a^{(a^*,a^*)}$.
		\item Let $a\in S$. Then $a$ is Drazin invertible if and only if $a$ is $(a^m,a^m)$-invertible for some positive integer $m$. In this case, $a^\mathrm{D}$ coincides with $a^{(a^m,a^m)}$, and $\operatorname{ind}(a)$ coincides with the least $m$ for which $a$ is $(a^m,a^m)$-invertible.
	\end{enumerate}
\end{lem}

Recently, Raki\'{c}, Din\v{c}i\'{c}, and Djordjevi\'{c} \cite{rakic_group_2014} extended the notion of core inverses to an arbitrary $^*$-ring case. In particular, they showed that core inverses belong to the class of $(b,c)$-inverses. In short, the core inverse of $a\in R$ coincides with $a^{(a,a^*)}$. See also \cite{gao_pseudo_2018,zhu_weighted_2019}.

Note that Drazin proposed an extended version of $(b,c)$-inverses named annihilator $(b,c)$-inverses (see \cite[Definition 6.2]{drazin_class_2012}). However, it was only defined over unital rings. To resolve this issue, we present an improved definition.
\begin{mdef}\label{def_ann}
	Let $a,b,c\in R$. We shall call $x\in R$ an annihilator $(b,c)$-inverse (ann-$(b,c)$-inverse for short) of $a$ if $x$ satisfies
	\begin{equation}\label{eq_def_ann}
	\mathrm{1)}~xax=x,\quad
	\mathrm{2)}~xab=b,\quad
	\mathrm{3)}~cax=c,\quad
	\mathrm{4)}~\prescript{\circ}{}{b}\subseteq\prescript{\circ}{}{x},\quad
	\mathrm{5)}~c^\circ\subseteq x^\circ.
	\end{equation}
\end{mdef}
It is easy to see that this modification makes no difference when $R$ is unital. In this case, the condition (\ref{eq_def_ann}.1) can be dropped since (\ref{eq_def_ann}.2) and (\ref{eq_def_ann}.4) imply (\ref{eq_def_ann}.1) (or see Corollary \ref{cor_sided_ann}). By Theorem \ref{thm_unique_ann} below, these five conditions determine $x$ uniquely when it exists. Thus, we denote by $a^{\circ(b,c)}$ the unique solution to \eqref{eq_def_ann} when $a$ is annihilator $(b,c)$-invertible.

This paper is organized as follows. In Section \ref{sec 2}, we first show that, contrary to one-sided $(b,c)$-inverses of elements taken from semigroups, an element in an arbitrary ring could have left (or right) annihilator $(b,c)$-inverses that are different from its annihilator $(b,c)$-inverse (see Theorem \ref{thm_sided_ann}). We then give some examples to illustrate this point. Several results related to one-sided annihilator $(b,c)$-inverses and its uniqueness are also presented. In Section \ref{sec 3}, we consider an intertwining relation (see Theorem \ref{thm_comm_ann}), namely, $yx_1=x_2y$, where $y\in R$, and $x_i \in R$ is the annihilator $(b_i,c_i)$-inverse of $a_i\in R$ for given $b_i,c_i\in R$, $i=1,2$. More specifically, we derive equivalent conditions for each equality below to hold.
\begin{equation}\label{eq_intro_intertwining}
\mathrm{1)}~ya_1x_1=a_2x_2y;\quad
\mathrm{2)}~yx_1a_1=x_2a_2y;\quad
\mathrm{3)}~ya_1x_1=x_2a_2y;\quad
\mathrm{4)}~yx_1a_1=a_2x_2y.
\end{equation}
Indeed, such equalities can be reinterpreted by $(b,c)$-inverses in semigroups. We also determine their equivalent conditions by only using the relations between $a_i,b_i,c_i$, $i=1,2$, and $y$ (see Theorems \ref{thm_comm} and \ref{thm_intertwine}). It might be emphasized that, to the best of our knowledge, these necessary and sufficient conditions seem not to have been noted before even in the $(b,c)$-inverse case. In addition, in the one-sided $(b,c)$-inverse case, we recognize that it can only give us a necessary (or sufficient) condition for intertwining property (see Remark \ref{rmk_intertwining_sided}). As applications, we study the absorption law and the reverse order law for annihilator $(b,c)$-inverses, and then for $(b,c)$-inverses in Section \ref{sec 4}. More precisely, we show that the absorption law holds, namely, $x_1+x_2=x_1(a_1+a_2)x_2$, where $x_1,x_2$ are taken to be the same as above, if $b_1=b_2$ and $c_1=c_2$ (see Theorem \ref{thm_adsorption_ann}). Also, we present equivalent conditions for the reverse order law to hold, namely, $(a_1a_2)^{\circ(b_2,c_1)}=a_2^{\circ(b_2,c_2)}a_1^{\circ(b_1,c_1)}$ (see Theorem \ref{thm_reverse_ann}). As corollaries, we propose some general, yet useful results (see Theorems \ref{thm_suff_reverse_ann} and \ref{thm_suff_reverse}). Finally, we close the paper with a discussion on Cline's formula.

\section{General results}\label{sec 2}

In this section, we obtain a number of properties of one-sided annihilator $(b,c)$-inverses. We first show the uniqueness of annihilator $(b,c)$-inverses.

\begin{thm}\label{thm_unique_ann}
	Let $a,b,c\in R$. Then $a$ has at most one ann-$(b,c)$-inverse.
\end{thm}
\begin{proof}
	Suppose that $x,y\in R$ both satisfy the condition \eqref{eq_def_ann}. From $xab=b=yab$, we have $(x-y)ab=0$. Also, $\prescript{\circ}{}{(ab)}\subseteq\prescript{\circ}{}{(ay)}$ since $\prescript{\circ}{}{b}\subseteq\prescript{\circ}{}{y}$. This yields $(x-y)ay=0$. Hence $y=yay=xay$. Similarly, we get $x=xax=xay$. Then $x=xay=y$. Thus $a$ has at most one ann-$(b,c)$-inverse.
\end{proof}

\begin{cor}\label{cor_relation}
	Let $R$ be a $^*$-ring, and let $a\in R$. Then $a$ is Moore--Penrose invertible if and only if $a$ is ann-$(a^*,a^*)$-invertible. In this case, $a^\dagger$ coincides with $a^{\circ(a^*,a^*)}$.
\end{cor}
\begin{proof}
	Suppose that $a$ is Moore--Penrose invertible. From Lemma \ref{lem_relation}, $a$ is $(a^*,a^*)$-invertible, and hence is ann-$(a^*,a^*)$-invertible. By Theorem \ref{thm_unique_ann}, $a^{\circ(a^*,a^*)}=a^{(a^*,a^*)}=a^\dagger$.
	
	Conversely, suppose that $a$ is ann-$(a^*,a^*)$-invertible. It is easy to verify that (\ref{eq_def_ann}.2), i.e., $xaa^*=a^*$, is equivalent to (\ref{eq_def_mp}.1) and (\ref{eq_def_mp}.4), i.e., $axa=a$ and $(xa)^*=xa$. Similarly, (\ref{eq_def_ann}.3) is equivalent to (\ref{eq_def_mp}.1) and (\ref{eq_def_mp}.3). The condition (\ref{eq_def_mp}.2) is obvious. Thus, $a$ is Moore--Penrose invertible with $a^\dagger=a^{\circ(a^*,a^*)}$.
\end{proof}

We now give the definition of one-sided annihilator $(b,c)$-inverses.

\begin{mdef}\label{def_sided_ann}
	Let $a,b,c\in R$. We call $x\in R$ a left annihilator $(b,c)$-inverse (lann-$(b,c)$-inverse for short) of $a$, if $x$ satisfies
	\begin{equation}\label{eq_def_lann}
	xab=b,\quad
	c^\circ\subseteq x^\circ.
	\end{equation}
	Dually, we call $y\in R$ a right annihilator $(b,c)$-inverse (rann-$(b,c)$-inverse for short) of $a$, if $y$ satisfies
	\begin{equation}\label{eq_def_rann}
	cay=c,\quad
	\prescript{\circ}{}{b}\subseteq\prescript{\circ}{}{y}.
	\end{equation}
\end{mdef}

Actually, a similarly definition appears to have been obtained by Ke, Vi\v{s}nji\'{c}, and Chen \cite[Definition 2.3]{ke_one-sided_2016}. It should be pointed out that left (resp. right) annihilator $(b,c)$-inverses given above by right (resp. left) annihilator in order to keep the names consistent with Drazin's definition of one-sided $(b,c)$-inverses (see \cite[Definition 1.2]{drazin_left_2016}).

\begin{mdef}[cf. {\cite[Definition 2.2]{drazin_left_2016}}]
	Let $a,b,c\in R$. For any given left (or right) ann-$(b,c)$-inverse $x\in R$ of $a$, we call $x$ is \textit{regular} if $x$ satisfies $xax=x$.
\end{mdef}

Note, this is a different definition from that of \textit{von Neumann regular}.

\begin{thm}\label{thm_sided_ann}
	Let $a,b,c\in R$, and suppose that $a$ is both left and right ann-$(b,c)$-invertible with a lann-$(b,c)$-inverse $x_l$ and a rann-$(b,c)$-inverse $x_r$. Then $a$ is ann-$(b,c)$-invertible with the ann-$(b,c)$-inverse $x_lax_r$.
\end{thm}
\begin{proof}
	Since $x_lab=b$ and $cax_r=c$, we conclude that $cax_rax_lab=cab$. Using $c^\circ\subseteq x_l^\circ$ and $\prescript{\circ}{}{b}\subseteq\prescript{\circ}{}{x_r}$, we see that $(x_lax_r)a(x_lax_r)=x_lax_r$. Next, from $x_lab=b$, we have $h(x_lab)=hb$ for all $h\in R$. By $\prescript{\circ}{}{b}\subseteq\prescript{\circ}{}{x_r}$, we get $h(x_lax_r)=hx_r$ for all $h\in R$. (Note, this is a technique that we will use over and over again.) Thus $\prescript{\circ}{}{b}\subseteq\prescript{\circ}{}{x_r}=\prescript{\circ}{}{(x_lax_r)}$. Then by taking $h=ca$, we get $ca(x_lax_r)=cax_r=c$. Similarly, we obtain $c^\circ\subseteq x_l^\circ=(x_lax_r)^\circ$ and $(x_lax_r)ab=x_lab=b$.
\end{proof}

This is noteworthy, because Drazin in \cite[Theorem 2.1]{drazin_left_2016} has proved that if $a\in S$ is both left and right $(b,c)$-invertible for given $b,c\in S$ then $a$ is $(b,c)$-invertible and its left and right $(b,c)$-inverse are unique (and are both equal to $a^{(b,c)}$). Recall Corollary \ref{cor_relation}. Theorem \ref{thm_sided_ann} may also remind many readers of a commonly known property of Moore--Penrose inverses,
\begin{equation}\label{eq_mp}
a^\dagger=a^{\{1,4\}}aa^{\{1,3\}},
\end{equation}
where $a$ is a Moore--Penrose invertible element in a $^*$-ring $R$, and $a^{\{1,4\}}\in R$ (resp. $a^{\{1,3\}}\in R$) is a solution satisfying (\ref{eq_def_mp}.1) and (\ref{eq_def_mp}.4) (resp. (\ref{eq_def_mp}.3)). However, we would like to emphasize that we could not obtain \eqref{eq_mp} by slightly modifying Theorem \ref{thm_sided_ann}, because one-sided annihilator $(a^*,a^*)$-inverses of $a$ is ``stronger'' than $a^{\{1,4\}}$ and $a^{\{1,3\}}$, i.e., left annihilator $(a^*,a^*)$-inverses of $a$ have to satisfy both conditions $xaa^*=a^*$ and $(a^*)^\circ\subseteq x^\circ$ while $a^{\{1,4\}}$ only requires $xaa^*=a^*$. For example, take $2\times2$ complex matrix $a=\begin{pmatrix}1&0\\0&0\end{pmatrix}$. Then there exits $x=\begin{pmatrix}1&0\\0&1\end{pmatrix}$ such that $xaa^*=a^*$ but not satisfy $(a^*)^\circ\subseteq x^\circ$.

To distinguish Theorem \ref{thm_sided_ann} from Drazin's result, we construct the following examples which are inspired by Johnson's work \cite{johnson_structure_1957}.

\begin{eg}\label{eg_sided_ann_1}
	Let $\mathscr{M}_3(\mathbb{Z}_2)$ be the ring of all $3\times3$ matrices with entries from the prime field of order 2, and let $R:=\{\alpha e_{11}+\beta e_{21}+\gamma e_{22}+\delta e_{31}:\alpha,\beta,\gamma,\delta\in\mathbb{Z}_2\}$ be a subring of $\mathscr{M}_3(\mathbb{Z}_2)$. Take $a=e_{11}+e_{22},b=c=e_{11}+e_{21}$. Then $x=e_{11}+e_{21}$ is the ann-$(b,c)$-inverse of $a$. Note that $y=e_{11}+e_{21}+e_{31}$ is a rann-$(b,c)$-inverse of $a$ and satisfies $y=yay$, but $y\neq x$.
\end{eg}

\begin{eg}\label{eg_sided_ann_2}
	Let $R$ be the same as above. Take $a=b=c=e_{22}$. Then $x=e_{22}$ is the ann-$(b,c)$-inverse of $a$. Note that $y=e_{22}+e_{31}$ is a rann-$(b,c)$-inverse of $a$, but $y\neq yay, y\neq x$.
\end{eg}

These examples also demonstrate that, for a given element $a\in R$, its right annihilator $(b,c)$-inverse, if exists, may be not unique (and even not regular when $a$ is annihilator $(b,c)$-invertible) in general. Accordingly, we may of course be interested in the relation between annihilator $(b,c)$-inverses and one-sided ones.

\begin{prop}\label{prop_sided_ann_1}
	Let $a,b,c\in R$, and suppose that $a$ is ann-$(b,c)$-invertible with a lann-$(b,c)$-inverse $x_l$ and a rann-$(b,c)$-inverse $x_r$. Then $a^{\circ(b,c)}=x_l$ if and only if $x_l\in Rx_r$. (And dually for $x_r$.)
\end{prop}
\begin{proof}
	Suppose that $x_l\in Rx_r$. Then there exist some $t\in R$ such that $tx_r=x_l$. From $x_lab=b$ and $\prescript{\circ}{}{b}\subseteq\prescript{\circ}{}{x_r}$, we have $tx_lax_r=tx_r=x_l$. Therefore,
	\[
	b=x_lab=\left(tx_lax_r\right)ab=\left(ta^{\circ(b,c)}\right)ab=t\left(a^{\circ(b,c)}ab\right)=tb=t\left(x_lab\right)=\left(tx_l\right)ab.
	\]
	On the other hand, $c^\circ\subseteq x_l^\circ\subseteq (tx_l)^\circ$, hence $tx_l$ is a lann-$(b,c)$-inverse of $a$. By Theorem \ref{thm_sided_ann}, $a^{\circ(b,c)}=(tx_l)ax_r=x_l$.
	
	Conversely, suppose that $a^{\circ(b,c)}=x_l$. We then have $x_l=a^{\circ(b,c)}=(x_la)x_r\in Rx_r$.
\end{proof}

Clearly, if $x_l\in Rc(\subseteq Rx_r)$, then $a^{\circ(b,c)}=x_l$. Therefore, Proposition \ref{prop_sided_ann_1} also shows that a $(b,c)$-invertible element in rings has a unique left (or right) $(b,c)$-inverse.

\begin{cor}\label{cor_sided_ann}
	Let $a,b,c\in R$. If there exist some $x\in R$ such that $xab=b,cax=c,\prescript{\circ}{}{b}\subseteq\prescript{\circ}{}{x}$, and $c^\circ\subseteq x^\circ$, then $a$ is ann-$(b,c)$-invertible with the ann-$(b,c)$-inverse $xax$. In this case, $xax=x$ if and only if $x\in Rx$ (or equivalently, $x\in xR$).
\end{cor}

\begin{prop}\label{prop_sided_ann_2}
	Let $a,b,c\in R$, and suppose that $a$ is ann-$(b,c)$-invertible with a lann-$(b,c)$-inverse $x_l$. Then $a^{\circ(a,b)}=x_lax_l$ if and only if $cax_l=c$. (And dually for rann-$(b,c)$-inverse.)
\end{prop}

In order to demonstrate proposition \ref{prop_sided_ann_2}, we begin by establishing the following lemma.

\begin{lem}\label{lem_prop_sided_ann_2}
	Let $a,b,c\in R$. If $a$ is ann-$(b,c)$-invertible with a lann-$(b,c)$-inverse $x_l$ then $\prescript{\circ}{}{b}\subseteq\prescript{\circ}{}{(x_lax_l)}$ and $x_lax_lax_l=x_lax_l$.
\end{lem}
\begin{proof}
	It is clear that $x_laa^{\circ(b,c)}h=x_lh$ for all $h\in R$, since $c^\circ\subseteq x_l^\circ$ and $caa^{\circ(b,c)}=c$. From $x_laa^{\circ(b,c)}=a^{\circ(b,c)}$ (note that $a^{\circ(b,c)}$ is also a rann-$(b,c)$-inverse of $a$), we get $a^{\circ(b,c)}h=x_lh$ for all $h\in R$. This yields $a^{\circ(b,c)}-x_l\in \prescript{\circ}{}{R}$. Now, we assume that $x_l=a^{\circ(b,c)}+\varepsilon$, where $\varepsilon\in\prescript{\circ}{}{R}$, i.e., $\varepsilon R=\{0\}$. We have $\prescript{\circ}{}{b}\subseteq\prescript{\circ}{}{(a^{\circ(b,c)})}\subseteq\prescript{\circ}{}{(x_lax_l)}$, since
	\[
	x_lax_l=\left(a^{\circ(b,c)}+\varepsilon\right)a\left(a^{\circ(b,c)}+\varepsilon\right)=a^{\circ(b,c)}+a^{\circ(b,c)}a\varepsilon.
	\]
	We also get
	\[
	x_lax_lax_l=x_la\left(a^{\circ(b,c)}+a^{\circ(b,c)}a\varepsilon\right)=a^{\circ(b,c)}+a^{\circ(b,c)}a\varepsilon=x_lax_l.\qedhere
	\]
\end{proof}

\begin{proof}[Proof of Proposition \ref{prop_sided_ann_2}]
	Suppose that $cax_l=c$. From Lemma \ref{lem_prop_sided_ann_2}, we see that $\prescript{\circ}{}{b}\subseteq\prescript{\circ}{}{(x_lax_l)}$ and $(x_lax_l)a(x_lax_l)=x_lax_lax_l=x_lax_l$. We also obtain that
	\begin{align*}
	&ca\left(x_lax_l\right)=\left(cax_l\right)ax_l=cax_l=c,\\
	&c^\circ\subseteq x_l^\circ\subseteq \left(x_lax_l\right)^\circ,\\
	&\left(x_lax_l\right)ab=x_la\left(x_lab\right)=x_lab=b.
	\end{align*}
	By Definition \ref{def_ann}, $x_lax_l$ is the ann-$(b,c)$-inverse of $a$.
	
	Conversely, suppose that $a^{\circ(a,b)}=x_lax_l$. From $x_lax_l=x_lax_lax_l$, we conclude that
	\[
	c=caa^{\circ(b,c)}=ca(x_lax_l)=ca(x_lax_lax_l)=(cax_lax_l)ax_l=(caa^{\circ(b,c)})ax_l=cax_l.\qedhere
	\]
\end{proof}

We observe that $cax_l=c$ is equivalent to $c\in Rx_l$ when $x_l$ is regular, i.e., $x_lax_l=x_l$. This leads to the following corollary (corresponding to Proposition \ref{prop_sided_ann_1}).

\begin{cor}
	Let $a,b,c\in R$, and suppose that $a$ is ann-$(b,c)$-invertible with a regular lann-$(b,c)$-inverse $x_l$. Then $a^{\circ(b,c)}=x_l$ if and only if $c\in Rx_l$. (And dually for regular rann-$(b,c)$-inverse.)
\end{cor}

Again, in view of the proof of Lemma \ref{lem_prop_sided_ann_2} above. When $R$ is \textit{left faithful}, namely, $\prescript{\circ}{}{R}=\{0\}$, we get $x_l=a^{\circ(b,c)}$ (since $\varepsilon\in\prescript{\circ}{}{R}$ implies $\varepsilon=0$). Then we immediately have the following proposition.

\begin{prop}\label{prop_sided_ann_3}
Let $R$ be left (resp. right) faithful. Let $a,b,c\in R$, and suppose that $a$ is ann-$(b,c)$-invertible. Then $a$ has a unique left (resp. right) ann-$(b,c)$-inverse, which is equal to $a^{\circ(b,c)}$.
\end{prop}

In particular, if $R$ is cofaithful (or, more specifically, is unital) and $a\in R$ is both left and right annihilator $(b,c)$-invertible with a left annihilator $(b,c)$-inverse $x_l$ and a right annihilator $(b,c)$-inverse $x_r$, then $a$ is annihilator $(b,c)$-invertible with $a^{\circ(b,c)}=x_l=x_r$.

\begin{prop}\label{prop_sided_ann_4}
	Let $a,b,c\in R$, and suppose that $a$ is lann-$(b,c)$-invertible with a lann-$(b,c)$-inverse $x_l$. If $x_la$ is Drazin invertible, then $a$ has a regular lann-$(b,c)$-inverse. (And dually for rann-$(b,c)$-inverse.)
\end{prop}
\begin{proof}
	Assume that $t\in R$ is the Drazin inverse of $x_la$. We now prove that $tx_l$ is the one required by this proposition. It is easy to see $(tx_l)a(tx_l)=(t(x_la)t)x_l=tx_l$. From
	\[
	b=x_lab=x_la(x_lab)=(x_la)^2b=\cdots=(x_la)^mb=t(x_la)^{m+1}b=tb,
	\]
	where $m$ is the Drazin index of $x_la$, we have $(tx_l)ab=tb=b$. Hence we complete the proof since $c^\circ\subseteq x_l^\circ\subseteq (tx_l)^\circ$. 
\end{proof}

\section{Intertwining properties}\label{sec 3}
Let $x_i \in S$ be the $(b_i,c_i)$-inverse of $a_i\in S$ for given $b_i,c_i\in S$, $i=1,2$. The equality
\begin{equation}\label{eq_intertwining}
yx_1=x_2y,\text{ where }y\in S,
\end{equation}
is known as an intertwining relation. Our main results in this section are Theorems \ref{thm_comm} and \ref{thm_intertwine}. We approach these results by first considering the case of annihilator $(b,c)$-inverses. Conditions $xab=b$ and $cax=c$ in \eqref{eq_def_ann} will play a key role in our proof. Some special cases of intertwining relations are also discussed.

First, we give the following facts which will be used often in the sequel. We omit its proof since the proof is straightforward.

\begin{fact}\label{fact}
The following statements hold.
\begin{enumerate}[i)]
	\item If $x\in S$ is the $(b,c)$-inverse of $a\in S$, then $bS=xS$ and $Sc=Sx$.
	\item If $x\in R$ is the ann-$(b,c)$-inverse of $a\in R$, then $\prescript{\circ}{}{x}=\prescript{\circ}{}{b}$ and $x^\circ=c^\circ$.
	\item Let $t\in S$, and let $a,x\in S$ such that $xax=x$. Then $t\in xS$ (resp. $t\in Sx$) if and only if $t=xat$ (resp. $t=tax$). In particular, $t\in axS$ (resp. $t\in Sxa$) if and only if $t=axt$ (resp. $t=txa$).
\end{enumerate}
\end{fact}

\begin{thm}\label{thm_comm_ann}
Let $a_i, b_i, c_i\in R$, and suppose that $a_i$ is ann-$(b_i,c_i)$-invertible with the ann-$(b_i,c_i)$-inverse $x_i$, $i=1,2$. Then, for any given $y\in R^1$, $yx_1=x_2y$ if and only if $c_2ya_1b_1=c_2a_2yb_1,yb_1\in x_2R$, and $c_2y\in Rx_1$ (or equivalently, $c_2ya_1b_1=c_2a_2yb_1,yb_1=x_2a_2yb_1$, and $c_2y=c_2ya_1x_1$).
\end{thm}
\begin{proof}
	Set $ya_1-a_2y=\varepsilon_1,yb_1-b_2y=\varepsilon_2,yc_1-c_2y=\varepsilon_3$, and $yx_1-x_2y=\tau$. We have
	\begin{align*}
	\left(yx_1-x_2y\right)a_1x_1&=yx_1a_1x_1-x_2\left(a_2y+\varepsilon_1\right)x_1\\
	&=yx_1-x_2a_2\left(x_2y+\tau\right)-x_2\varepsilon_1x_1\\
	&=\left(yx_1-x_2y\right)-x_2a_2\tau-x_2\varepsilon_1x_1,
	\end{align*}
	that is,
	\begin{equation}\label{eq_thm_comm_ann_1}
	\tau=\tau a_1x_1+x_2a_2\tau+x_2\varepsilon_1x_1.
	\end{equation}
	Also, we have
	\begin{align*}
	\left(yx_1-x_2y\right)a_1b_1&=yx_1a_1b_1-x_2\left(a_2y+\varepsilon_1\right)b_1\\
	&=yb_1-x_2a_2\left(b_2y+\varepsilon_2\right)-x_2\varepsilon_1b_1\\
	&=\left(yb_1-b_2y\right)-x_2a_2\varepsilon_2-x_2\varepsilon_1b_1,
	\end{align*}
	that is,
	\begin{equation}\label{eq_thm_comm_ann_2}
	\tau a_1b_1=\varepsilon_2-x_2a_2\varepsilon_2-x_2\varepsilon_1b_1.
	\end{equation}
	Similarly, we get
	\begin{equation}\label{eq_thm_comm_ann_3}
	c_2a_2\tau=\varepsilon_3-\varepsilon_3a_1x_1-c_2\varepsilon_1x_1.
	\end{equation}
	Consider the following statements.
	\begin{equation}\label{eq_thm_comm_ann_4}
	\mathrm{1)}~\begin{cases}
	x_2\varepsilon_1x_1=0,\\
	\tau a_1x_1=0,\\
	x_2a_2\tau=0.
	\end{cases}\qquad
	\mathrm{2)}~\begin{cases}
	x_2\varepsilon_1x_1=0,\\
	\tau a_1b_1=0,\\
	c_2a_2\tau=0.
	\end{cases}\qquad
	\mathrm{3)}~\begin{cases}
	c_2\varepsilon_1b_1=0,\\
	\varepsilon_2=x_2a_2\varepsilon_2,\\
	\varepsilon_3=\varepsilon_3a_1x_1.
	\end{cases}
	\end{equation}
	By \eqref{eq_thm_comm_ann_1}, it is easy to see that $\tau=0$ is equivalent to (\ref{eq_thm_comm_ann_4}.1). Next, $\tau a_1x_1=0$ and $x_2a_2\tau=0$ are equivalent, respectively, to $\tau a_1b_1=0$ and $c_2a_2\tau=0$ since $\prescript{\circ}{}{x_1}=\prescript{\circ}{}{b_1}$ and $x_2^\circ=c_2^\circ$. Thus (\ref{eq_thm_comm_ann_4}.1)$\iff$(\ref{eq_thm_comm_ann_4}.2). Also, the following equalities are equivalent.
	\[
	x_2\varepsilon_1x_1=0,\quad
	x_2\varepsilon_1b_1=0,\quad
	c_2\varepsilon_1x_1=0,\quad
	c_2\varepsilon_1b_1=0.
	\]
	Suppose that (\ref{eq_thm_comm_ann_4}.2) hold. Note that $x_2\varepsilon_1x_1=0$ implies $x_2\varepsilon_1b_1=0$. From \eqref{eq_thm_comm_ann_2}, equalities $x_2\varepsilon_1b_1=0$ and $\tau a_1b_1=0$ give $\varepsilon_2=x_2a_2\varepsilon_2$. Similarly, we have $\varepsilon_3=\varepsilon_3a_1x_1$. The proof of (\ref{eq_thm_comm_ann_4}.3)$\implies$(\ref{eq_thm_comm_ann_4}.2) is almost the same as (\ref{eq_thm_comm_ann_4}.2)$\implies$(\ref{eq_thm_comm_ann_4}.3). Therefore, (\ref{eq_thm_comm_ann_4}.2)$\iff$(\ref{eq_thm_comm_ann_4}.3).
	
	Combining with $ya_1-a_2y=\varepsilon_1$, equality $c_2\varepsilon_1b_1=0$ can be equivalently written as $c_2(ya_1-a_2y)b_1=0$, that is, $c_2ya_1b_1=c_2a_2yb_1$. In this case, $\varepsilon_2=x_2a_2\varepsilon_2$ is equivalent to $(yb_1-b_2y)=x_2a_2(yb_1-b_2y)$, that is, $yb_1=x_2a_2yb_1$. By Fact \ref{fact}.iii, $\varepsilon_2=x_2a_2\varepsilon_2$ if and only if $yb_1\in x_2R$. Similarly, $\varepsilon_3=\varepsilon_3a_1x_1$ is equivalent to $c_2y\in Rx_1$.
\end{proof}

From $xab=b$ (resp. $cax=c$), we see $bR\subseteq xR$ (resp. $Rc\subseteq Rx$). We then have the following corollary.

\begin{cor}\label{cor_comm_ann}
Let $a_i, b_i, c_i\in R$, and suppose that $a_i$ is ann-$(b_i,c_i)$-invertible with the ann-$(b_i,c_i)$-inverse $x_i$, $i=1,2$. For any given $y\in R^1$, if $c_2ya_1b_1=c_2a_2yb_1,yb_1\in b_2R$, and $c_2y\in Rc_1$, then $yx_1=x_2y$.
\end{cor}

We now reveal the relation between Drazin inverses and annihilator $(b,c)$-inverses.

\begin{cor}
Let $a\in R$. Then $a$ is Drazin invertible if and only if $a$ is ann-$(a^m,a^m)$-invertible for some positive integer $m$. In this case, $a^\mathrm{D}$ coincides with $a^{\circ(a^m,a^m)}$, and $\operatorname{ind}(a)$ coincides with the least $m$ for which $a$ is ann-$(a^m,a^m)$-invertible.
\end{cor}
\begin{proof}
Suppose that $a$ is Drazin invertible. Set $n=\operatorname{ind}(a)$. Note that $a^n(a^\mathrm{D})^{n+1}=a^\mathrm{D}=(a^\mathrm{D})^{n+1}a^n$ implies $\prescript{\circ}{}{(a^n)}\subseteq\prescript{\circ}{}{(a^\mathrm{D})}$ and $(a^n)^\circ\subseteq(a^\mathrm{D})^\circ$. Thus, it is easy to verify that $a$ is ann-$(a^n,a^n)$-invertible with $a^{\circ(a^n,a^n)}=a^\mathrm{D}$.

Conversely, suppose that $a$ is ann-$(a^m,a^m)$-invertible. Since $aa^m=a^ma=a^{m+1}\in a^mR\cap Ra^m$, Corollary \ref{cor_comm_ann} shows that $a^{\circ(a^m,a^m)}$ commutes with $a$. Hence, $a$ is Drazin invertible with $a^\mathrm{D}=a^{\circ(a^m,a^m)}$.

Let $m$ be the least positive integer for which $a$ is ann-$(a^m,a^m)$-invertible. The ``only if'' part gives $m\leqslant\operatorname{ind}(a)$. On the other hand, the ``if'' part gives $\operatorname{ind}(a)\leqslant m$. Therefore, such $m$ coincides with $\operatorname{ind}(a)$.
\end{proof}

Theorem \ref{thm_comm_ann} also brings us to an important theorem as follows.

\begin{thm}\label{thm_comm}
Let $a_i, b_i, c_i\in S$, and suppose that $a_i$ is $(b_i,c_i)$-invertible with the $(b_i,c_i)$-inverse $x_i$, $i=1,2$. Then, for any given $y\in S^1$, $yx_1=x_2y$ if and only if $c_2ya_1b_1=c_2a_2yb_1,yb_1\in b_2S$, and $c_2y\in Sc_1$ (or equivalently, $c_2ya_1b_1=c_2a_2yb_1,yb_1=x_2a_2yb_1$, and $c_2y=c_2ya_1x_1$).
\end{thm}
\begin{proof}
	Fact \ref{fact} shows $yb_1=x_2a_2yb_1$ in this case can be equivalently written as $yb_1\in x_2S=b_2S$. Similarly, $c_2y=c_2ya_1x_1$ is equivalent to $c_2y\in Sc_1$.
	
	The ``if'' part follows from $yx_1=x_2a_2yx_1=x_2ya_1x_1=x_2y$ by using $x_1\in b_1S$ and $x_2\in Sc_2$. For the ``only if'' part, we see that $c_2a_2(yx_1)a_1b_1=c_2a_2(x_2y)a_1b_1$. Thus $c_2a_2yb_1=c_2ya_1b_1$. On the other hand,
	\[
	yb_1=yx_1a_1b_1=x_2ya_1b_1=x_2a_2x_2ya_1b_1=x_2a_2yx_1a_1b_1=x_2a_2yb_1.
	\]
	Similarly, we get $c_2y=c_2ya_1x_1$.
\end{proof}

\begin{rmk}\label{rmk_intertwining_sided}
According to the proof of Theorem \ref{thm_comm}, if $x_1$ is a right $(b_1,c_1)$-inverse of $a_1$ and $x_2$ is a left $(b_2,c_2)$-inverse of $a_2$ then $c_2ya_1b_1=c_2a_2yb_1,yb_1=x_2a_2yb_1, c_2y=c_2ya_1x_1$ is sufficient for $yx_1=x_2y$ (see e.g. \cite[Theorem 5.1]{drazin_left_2016}, or \cite{drazin_commuting_2013}). Conversely, if $x_1$ is a regular, i.e., $x_1a_1x_1=x_1$, left $(b_1,c_1)$-inverse of $a_1$ and $x_2$ is a regular right $(b_2,c_2)$-inverse of $a_2$ then $c_2ya_1b_1=c_2a_2yb_1,yb_1=x_2a_2yb_1, c_2y=c_2ya_1x_1$ is necessary for $yx_1=x_2y$.

In particular, take $a_1=a_2=a\in S$ and $y=1$. Then $x_i$, the $(b_i,c_i)$-inverse of $a$, $i=1,2$, coincide with each other if and only if $b_1S=b_2S$ and $Sc_1=Sc_2$ (cf. \cite[Remark 2.2.i]{boasso_b_2017}). The proof is close in spirit to that of Corollary \ref{cor_intertwine}.
\end{rmk}

Recall that $\mathcal{F}$ is a \textit{left} (resp. \textit{right}) \textit{centralizer} on $S$ if $\mathcal{F}$ is a map of $S$ to $S$ such that $\mathcal{F}(x)y=\mathcal{F}(xy)$ (resp. $x\mathcal{F}(y)=\mathcal{F}(xy)$) for all $x,y\in S$.

\begin{cor}[cf. {\cite[Lemma 3.6]{xu_centralizers_2019}}]
Let $a_i, b_i, c_i\in S$, and suppose that $a_i$ is $(b_i,c_i)$-invertible with the $(b_i,c_i)$-inverse $x_i$, $i=1,2$. For any $y\in S$ that satisfies $ya_1=a_2y$, if there exist a right centralizer $\mathcal{F}$ and a left centralizer $\mathcal{G}$ on $S$ such that $yb_1=\mathcal{F}(b_2y)$ and $c_2y=\mathcal{G}(yc_1)$, then $yx_1=x_2y$.
\end{cor}

For more applications of centralizers to generalized inverses, we refer readers to \cite{zhu_centralizers_2014,zhu_centralizers_2017,xu_centralizers_2019}. Likewise, our results in this section seem to propose a new approach to study $m$-EP elements in $^*$-rings. (We say $a$ is $m$-EP if $a$ is both Moore--Penrose and Drazin invertible, and its Moore--Penrose inverse commutes with $a^m$, where $m=\operatorname{ind}(a)$ \cite[Lemma 3.6]{zou_characterizations_2018}.) However, we will not discuss it here, since it may stray from the main point of this paper.

Motivated by \cite[Theorem 4.1]{drazin_bicommuting_2017}, we consider some special cases of intertwining relations.

\begin{thm}\label{thm_intertwine_ann}
Let $a_i, b_i, c_i\in R$, and suppose that $a_i$ is ann-$(b_i,c_i)$-invertible with the ann-$(b_i,c_i)$-inverse $x_i$, $i=1,2$. Then, for any given $y\in R^1$,
\begin{enumerate}[i)]
	\item $ya_1x_1=a_2x_2y\iff ya_1b_1\in a_2x_2R\text{ and }c_2y\in Rx_1$.
	\item $yx_1a_1=x_2a_2y\iff yb_1\in x_2R\text{ and }c_2a_2y\in Rx_1a_1$.
	\item $ya_1x_1=x_2a_2y\iff ya_1b_1\in x_2R\text{ and }c_2a_2y\in Rx_1$.
	\item $yx_1a_1=a_2x_2y\iff yb_1\in a_2x_2R\text{ and }c_2y\in Rx_1a_1$.
\end{enumerate}
\end{thm}
\begin{proof}
	Part i): Set $ya_1x_1-a_2x_2y=\tau_1$. From
	\[
	(ya_1x_1-a_2x_2y)a_1x_1=ya_1x_1a_1x_1-a_2x_2(a_2x_2y+\tau_1)=(ya_1x_1-a_2x_2y)-a_2x_2\tau_1,
	\]
	we have
	\begin{equation}\label{eq_thm_intertwine_ann}
	\tau_1=\tau_1a_1x_1+a_2x_2\tau_1.
	\end{equation}
	Hence, we obtain that $\tau_1=0$ is equivalent to $\tau_1a_1x_1=0$ and $a_2x_2\tau_1=0$. On the other hand, $(a_2x_2)^\circ\subseteq c_2^\circ$ since $c_2a_2x_2=c_2$. This yields $x_2^\circ\subseteq (a_2x_2)^\circ\subseteq c_2^\circ\subseteq x_2^\circ$, that is, $(a_2x_2)^\circ=c_2^\circ$. Combining with $\prescript{\circ}{}{x_1}=\prescript{\circ}{}{b_1}$, we see that $\tau_1a_1x_1=0$ and $a_2x_2\tau_1=0$ are equivalent, respectively, to $\tau_1a_1b_1=0$ and $c_2\tau_1=0$. Next, from
	\[
	\tau_1a_1b_1=(ya_1x_1-a_2x_2y)a_1b_1=ya_1b_1-a_2x_2ya_1b_1,
	\]
	$\tau_1a_1x_1=0$ if and only if $ya_1b_1=a_2x_2ya_1b_1$. Also, $ya_1b_1=a_2x_2ya_1b_1$ can be equivalently written as $ya_1b_1\in a_2x_2R$ (see Fact \ref{fact}.iii). Similarly, $c_2\tau_1=0$ if and only if $c_2y\in Rx_1$. Eventually, $ya_1x_1=a_2x_2y$ is equivalent to $ya_1b_1\in a_2x_2R$ and $c_2y\in Rx_1$.
	
	Part ii): Set $yx_1a_1-x_2a_2y=\tau_2$. From $(yx_1a_1-x_2a_2y)x_1a_1=yx_1a_1-x_2a_2(x_2a_2y+\tau_2)$, we have $\tau_2=\tau_2x_1a_1+x_2a_2\tau_2$. Next, $\prescript{\circ}{}{(x_1a_1)}\subseteq\prescript{\circ}{}{b_1}$ since $x_1a_1b_2=b_1$. This yields $\prescript{\circ}{}{x_1}\subseteq\prescript{\circ}{}{(x_1a_1)}\subseteq\prescript{\circ}{}{b_1}\subseteq\prescript{\circ}{}{x_1}$, that is, $\prescript{\circ}{}{(x_1a_1)}=\prescript{\circ}{}{b_1}$. Hence, we obtain that $\tau_2=0$ is equivalent to $\tau_2b_1=0$ and $c_2a_2\tau_2=0$. By simplifying, we get $yb_1=x_2a_2yb_1$ and $c_2a_2y=c_2a_2yx_1a_1$, that is, $yb_1\in x_2R$ and $c_2a_2y\in Rx_1a_1$.
	
	Part iii): Set $ya_1x_1-x_2a_2y=\tau_3$. From $(ya_1x_1-x_2a_2y)a_1x_1=ya_1x_1-x_2a_2(x_2a_2y+\tau_3)$, we have $\tau_3=\tau_3a_1x_1+x_2a_2\tau_3$. Hence, we obtain that $\tau_3=0$ is equivalent to $\tau_3a_1b_1=0$ and $c_2a_2\tau_3=0$. By simplifying, we get $ya_1b_1=x_2a_2ya_1b_1$ and $c_2a_2y=c_2a_2ya_1x_1$, that is, $ya_1b_1\in x_2R$ and $c_2a_2y\in Rx_1$.
	
	Part iv): Set $yx_1a_1-a_2x_2y=\tau_4$. From $(yx_1a_1-a_2x_2y)x_1a_1=yx_1a_1-a_2x_2(a_2x_2y+\tau_4)$, we have $\tau_4=\tau_4x_1a_1+a_2x_2\tau_4$. As in the proof of Parts i) and ii), we see that $(a_2x_2)^\circ=c_2^\circ$ and $\prescript{\circ}{}{(x_1a_1)}=\prescript{\circ}{}{b_1}$. Hence, we obtain that $\tau_4=0$ is equivalent to $\tau_4b_1=0$ and $c_2\tau_4=0$. By simplifying, we get $yb_1=a_2x_2yb_1$ and $c_2y=c_2yx_1a_1$, that is, $yb_1\in a_2x_2R$ and $c_2y\in Rx_1a_1$.
\end{proof}

Again, from $xab=b$ (resp. $cax=c$), we see $abR\subseteq axR$ (resp. $Rca\subseteq Rxa$).

\begin{cor}\label{cor_intertwine_ann}
Let $a_i, b_i, c_i\in R$, and suppose that $a_i$ is ann-$(b_i,c_i)$-invertible with the ann-$(b_i,c_i)$-inverse $x_i$, $i=1,2$. For any given $y\in R^1$,
\begin{enumerate}[i)]
	\item if $ya_1b_1\in a_2b_2R$ and $c_2y\in Rc_1$, then $ya_1x_1=a_2x_2y$.
	\item if $yb_1\in b_2R$ and $c_2a_2y\in Rc_1a_1$, then $yx_1a_1=x_2a_2y$.
	\item if $ya_1b_1\in b_2R$ and $c_2a_2y\in Rc_1$, then $ya_1x_1=x_2a_2y$.
	\item if $yb_1\in a_2b_2R$ and $c_2y\in Rc_1a_1$, then $yx_1a_1=a_2x_2y$.
\end{enumerate}
\end{cor}

\begin{thm}\label{thm_intertwine}
Let $a_i, b_i, c_i\in S$, and suppose that $a_i$ is $(b_i,c_i)$-invertible with the $(b_i,c_i)$-inverse $x_i$, $i=1,2$. Then, for any given $y\in S^1$,
\begin{enumerate}[i)]
	\item $ya_1x_1=a_2x_2y\iff ya_1b_1\in a_2b_2S\text{ and }c_2y\in Sc_1$.
	\item $yx_1a_1=x_2a_2y\iff yb_1\in b_2S\text{ and }c_2a_2y\in Sc_1a_1$.
	\item $ya_1x_1=x_2a_2y\iff ya_1b_1\in b_2S\text{ and }c_2a_2y\in Sc_1$.
	\item $yx_1a_1=a_2x_2y\iff yb_1\in a_2b_2S\text{ and }c_2y\in Sc_1a_1$.
\end{enumerate}
\end{thm}
\begin{proof}
	The ``only if'' part is not easy to see. We give a short explanation. Suppose that, for instance, $ya_1x_1=a_2x_2y$. Then $ya_1b_1=ya_1x_1a_1b_1=a_2x_2ya_1b_1$. By Fact \ref{fact}.iii, we get $ya_1b_1\in a_2x_2S$. Thus $ya_1b_1\in a_2b_2S$ since $b_2S=x_2S$. Similarly, we have $c_2y\in Sc_1$. We now focus on the ``if'' part. From $x_1\in b_1S$ and $x_2\in Sc_2$, there exist $v,w\in S$ such that $x_1=b_1v$ and $x_2=wc_2$.
	
	Part i): First, $b_2S=x_2S$ and $Sc_1=Sx_1$ show that $ya_1b_1\in a_2x_2S$ and $c_2y\in Sx_1$. Thus, there exist $p_1,q_1\in S$ that satisfy $ya_1b_1=a_2x_2p_1$ and $c_2y=q_1x_1$. Then $c_2p_1=c_2a_2x_2p_1=c_2ya_1b_1=q_1x_1a_1b_1=q_1b_1$. We conclude that
	\[
	ya_1x_1=ya_1b_1v=a_2x_2p_1v=a_2wc_2p_1v=a_2wq_1b_1v=a_2wq_1x_1=a_2wc_2y=a_2x_2y.
	\]

	Part ii): There exist $p_2,q_2\in S$ that satisfy $yb_1=x_2p_2$ and $c_2a_2y=q_2x_1a_1$. Then $c_2p_2=c_2a_2x_2p_2=c_2a_2yb_1=q_2x_1a_1b_1=q_2b_1$. We conclude that
	\[
	yx_1a_1=yb_1va_1=x_2p_2va_1=wc_2p_2va_1=wq_2b_1va_1=wq_2x_1a_1=wc_2a_2y=x_2a_2y.
	\]

	Part iii): There exist $p_3,q_3\in S$ that satisfy $ya_1b_1=x_2p_3$ and $c_2a_2y=q_3x_1$. Then $c_2p_3=c_2a_2x_2p_3=c_2a_2ya_1b_1=q_3x_1a_1b_1=q_3b_1$. We conclude that
	\[
	ya_1x_1=ya_1b_1v=x_2p_3v=wc_2p_3v=wq_3b_1v=wq_3x_1=wc_2a_2y=x_2a_2y.
	\]

	Part iv): There exist $p_4,q_4\in S$ that satisfy $yb_1=a_2x_2p_4$ and $c_2y=q_4x_1a_1$. Then $c_2p_4=c_2a_2x_2p_4=c_2yb_1=q_4x_1a_1b_1=q_4b_1$. We conclude that
	\begin{align*}
	yx_1a_1&=yb_1va_1=a_2x_2p_4va_1=a_2wc_2p_4va_1=a_2wq_4b_1va_1\\
	&=a_2wq_4x_1a_1=a_2wc_2y=a_2x_2y.\qedhere
	\end{align*}
\end{proof}

\begin{cor}\label{cor_intertwine}
	Let $a_i, b_i, c_i\in S$, and suppose that $a_i$ is $(b_i,c_i)$-invertible with the $(b_i,c_i)$-inverse $x_i$, $i=1,2$. Then the following statements are equivalent.
	\begin{align*}
	&\mathrm{i)}~x_1a_1=a_2x_2.&\quad
	&\mathrm{ii)}~a_2b_2S=b_1S\text{ and }Sc_1a_1=Sc_2.\\
	&\mathrm{iii)}~a_2b_2\in b_1S\text{ and }c_1a_1\in Sc_2.&\quad
	&\mathrm{iv)}~b_1\in a_2b_2S\text{ and }c_2\in Sc_1a_1.
	\end{align*}
\end{cor}

We present here a brief proof. ``i)$\iff$iii)'' follows from Theorem \ref{thm_intertwine}.iii. Also, ``i)$\iff$iv)'' follows from Theorem \ref{thm_intertwine}.iv. Then, ``i)$\implies$iii)'' and ``i)$\implies$iv)'' show ``i)$\implies$ii)''. On the other hand, note that $b_1=x_1a_1b_1\in x_1S=b_1S$. Similarly, $c_2\in Sc_2$. (See also \cite[Remark 2.2.iii]{boasso_b_2017}.) Therefore, ``iv)$\implies$i)'' gives ``ii)$\implies$i)''.

\section{Absorption laws and reverse order laws}\label{sec 4}

For any given invertible elements $\alpha,\beta$ in a unital ring, the equality
\begin{equation}
\alpha^{-1}+\beta^{-1}=\alpha^{-1}\left(\alpha+\beta\right)\beta^{-1}=\beta^{-1}\left(\alpha+\beta\right)\alpha^{-1}
\end{equation}
is known as the absorption law, and the equality
\begin{equation}
(\alpha\beta)^{-1}=\beta^{-1}\alpha^{-1}
\end{equation}
is known as the reverse order law. Using the conclusions in Section \ref{sec 3}, we can state and prove very general results related to the reverse order law for annihilator $(b,c)$-inverses.

\begin{lem}\label{lem_adsorption_sided_ann}
	Let $a_i, b_i, c_i\in R, i=1,2$. Suppose that $a_1$ is lann-$(b_1,c_1)$-invertible with a lann-$(b_1,c_1)$-inverse $x_1$, and $a_2$ is rann-$(b_2,c_2)$-invertible with a rann-$(b_2,c_2)$-inverse $x_2$. If $b_1=b_2$, then $x_1a_1x_2-x_2\in R^\circ$. Dually, if $c_1=c_2$, then $x_1a_2x_2-x_1\in \prescript{\circ}{}{R}$.
\end{lem}
\begin{proof}
	It is easy to see from $x_1a_1b_1=b_1$ and $\prescript{\circ}{}{b_1}=\prescript{\circ}{}{b_2}\subseteq\prescript{\circ}{}{x_2}$ that $h(x_1a_1x_2-x_2)=0$ for all $h\in R$. Therefore, $x_1a_1x_2-x_2\in R^\circ$.
\end{proof}

\begin{lem}\label{lem_adsorption_ann}
	Let $a_i, b_i, c_i\in R$, and suppose that $a_i$ is ann-$(b_i,c_i)$-invertible with the ann-$(b_i,c_i)$-inverse $x_i$, $i=1,2$. If $b_1=b_2$, then $x_1a_1x_2=x_2$ and $x_2a_2x_1=x_1$.
	
	Dually, if $c_1=c_2$, then $x_1a_2x_2=x_1$ and $x_2a_1x_1=x_2$.
\end{lem}
\begin{proof}
	Lemma \ref{lem_adsorption_sided_ann} gives $x_2a_2(x_1a_1x_2-x_2)=0$. This yields $x_2a_2x_1a_1x_2=x_2$. We now prove $x_2a_2x_1a_1x_2=x_1a_1x_2$. Indeed, $x_2a_2x_1a_1b_1=x_2a_2b_1=b_1=x_1a_1b_1$ which shows that $x_2a_2x_1a_1x_2=x_1a_1x_2$ as we require. Similarly, we get $x_2a_2x_1=x_1$.
\end{proof}

\begin{thm}\label{thm_adsorption_ann}
	Let $a_1, a_2, b, c\in R$, and suppose that $a_1,a_2$ are both ann-$(b,c)$-invertible with respective ann-$(b,c)$-inverses $x_1,x_2$. Then $x_1+x_2=x_1(a_1+a_2)x_2=x_2(a_1+a_2)x_1$.
\end{thm}

\begin{thm}[{cf. \cite[Theorem 2.3]{chen_reverse_2017}}]\label{thm_reverse_ann}
	Let $a_i, b_i, c_i\in R$, and suppose that $a_i$ is ann-$(b_i,c_i)$-invertible with the ann-$(b_i,c_i)$-inverse $x_i$, $i=1,2$. Then $a_1a_2$ is ann-$(b_2,c_1)$-invertible with $(a_1a_2)^{\circ(b_2,c_1)}=x_2x_1$ if and only if $x_2x_1a_1a_2b_2=b_2$ and $c_1a_1a_2x_2x_1=c_1$.
\end{thm}
\begin{proof}
	The ``only if'' part is obvious. We now prove the ``if'' part. Note that $c_1^\circ\subseteq x_1^\circ\subseteq (x_2x_1)^\circ,\prescript{\circ}{}{b_2}\subseteq\prescript{\circ}{}{x_2}\subseteq\prescript{\circ}{}{(x_2x_1)}$, and $x_2x_1=x_2x_1a_1x_1\in x_2x_1R$. Combining with Corollary \ref{cor_sided_ann}, we complete the proof.
\end{proof}

\begin{thm}\label{thm_suff_reverse_ann}
	Let $a_i, b_i, c_i\in R$, and suppose that $a_i$ is ann-$(b_i,c_i)$-invertible with the ann-$(b_i,c_i)$-inverse $x_i$, $i=1,2$. If one of the following statements holds, then $a_1a_2$ is ann-$(b_2,c_1)$-invertible with $(a_1a_2)^{\circ(b_2,c_1)}=x_2x_1$.
	\begin{enumerate}[i)]
		\item $x_1a_1=a_1x_1$ and $c_1=c_2$.
		\item $x_2a_2=a_2x_2$ and $b_1=b_2$.
		\item $x_1a_1=a_2x_2$.
	\end{enumerate}
\end{thm}
\begin{proof}
	i) holds: By Lemma \ref{lem_adsorption_ann}, $x_2a_2b_2=b_2$ implies $(x_2a_1x_1)a_2b_2=b_2$. Therefore, $x_2x_1a_1a_2b_2=b_2$. On the other hand, Theorem \ref{thm_comm_ann} shows $c_1a_1\in Rx_1$ since $x_1a_1=a_1x_1$. Assume that $c_1a_1=tx_1$, where $t\in R$. Then
	\[
	c_2=c_1=c_1a_1x_1=tx_1x_1=t(x_1a_2x_2)x_1=(tx_1)a_2x_2x_1=(c_1a_1)a_2x_2x_1=c_2a_1a_2x_2x_1.
	\]
	The result then follows from Theorem \ref{thm_reverse_ann}.
	
	ii) holds: The proof is analogous to that of the ``i)'' case.
	
	iii) holds: Clearly, $c_1=c_1a_1x_1=c_1a_1(x_1a_1x_1)=c_1a_1a_2x_2x_1$ and $b_2=x_2a_2b_2=(x_2a_2x_2)a_2b_2=x_2x_1a_1a_2b_2$, as required to complete the proof.
\end{proof}

By Corollary \ref{cor_intertwine_ann}, $a_2b_2\in b_1R$ and $c_1a_1\in Rc_2$ show $x_1a_1=a_2x_2$. Therefore, Theorem \ref{thm_suff_reverse_ann} gives us an extension of \cite[Theorems 3.5, 3.8, and 3.12]{xu_centralizers_2019}.

\begin{thm}\label{thm_suff_reverse}
	Let $a_i, b_i, c_i\in S$, and suppose that $a_i$ is $(b_i,c_i)$-invertible with the $(b_i,c_i)$-inverse $x_i$, $i=1,2$. If one of the following statements holds, then $a_1a_2$ is $(b_2,c_1)$-invertible with $(a_1a_2)^{(b_2,c_1)}=x_2x_1$.
	\begin{enumerate}[i)]
		\item $x_1a_1=a_1x_1$ and $c_1=c_2$.
		\item $x_2a_2=a_2x_2$ and $b_1=b_2$.
		\item $x_1a_1=a_2x_2$.
	\end{enumerate}
\end{thm}

Since it can be proved in a similar way as Theorem \ref{thm_suff_reverse_ann} (by using Theorem \ref{thm_comm}, \cite[Theorem 2.3]{chen_reverse_2017}, and an improved version of \cite[Lemma 2.1]{xu_centralizers_2019}), we skip the proof and refer readers to. Recall Corollary \ref{cor_intertwine}. These statements above are equivalent, respectively, to the following conditions.
\begin{enumerate}[i)]
	\item $a_1b_1S=b_1S,Sc_1a_1=Sc_1$, and $c_1=c_2$.
	\item $a_2b_2S=b_2S,Sc_2a_2=Sc_2$, and $b_1=b_2$.
	\item $a_2b_2S=b_1S$ and $Sc_1a_1=Sc_2$.
\end{enumerate}
Accordingly, there is an interesting question implicit here.

\begin{op}
	Characterize, if possible, equivalent conditions for the reverse order law, $(a_1a_2)^{(b_2,c_1)}=a_2^{(b_2,c_2)}a_1^{(b_1,c_1)}$, to hold by only using the relations between $a_i,b_i,c_i$, $i=1,2$.
\end{op}

\section{Cline's formula}\label{sec 5}

In this section, we generalize Cline's formula to the case of annihilator $(b,c)$-inverses. The ideas arose most directly from \cite[Section 6]{drazin_left_2016} and \cite[Theorem 2.1]{mosic_note_2015}.

\begin{thm}\label{thm_cline_ann}
	Let $a_1,a_2,b,c\in R$, and suppose that $(a_1a_2)^{n+1}$ is ann-$(b,c)$-invertible with the ann-$(b,c)$-inverse $x$ for some positive integer $n$, then $(a_2a_1)^{n}$ is ann-$(a_2b,ca_1)$-invertible with the ann-$(a_2b,ca_1)$-inverse $a_2xa_1$.
\end{thm}
\begin{proof}
	We conclude that
	\begin{align*}
	&(a_2xa_1)(a_2a_1)^n(a_2xa_1)=a_2x(a_1a_2)^{n+1}xa_1=a_2xa_1,\\
	&(a_2xa_1)(a_2a_1)^n(a_2b)=a_2x_1(a_1a_2)^{n+1}b=a_2b,\\
	&(ca_1)(a_2a_1)^n(a_2xa_1)=c(a_1a_2)^{n+1}xa_1=ca_1,\\
	&\prescript{\circ}{}{(a_2b)}=\prescript{\circ}{}{(a_2x)}\subseteq\prescript{\circ}{}{(a_2xa_1)},\\
	&(ca_1)^\circ=(xa_1)^\circ\subseteq(a_2xa_1)^\circ.\qedhere
	\end{align*}
\end{proof}

For notational convenience, the bicommutant of $a\in R$ is defined as follows.
\[
\operatorname{comm}^2\{a\}:=\{r\in R:rh=hr\text{ for all }h\in R\text{ such that }ha=ah\}.
\]

\begin{prop}\label{prop_cline_ann}
	Let $a_1,a_2,b,c\in R$, and suppose that $(a_1a_2)^{n+1}$ is ann-$(b,c)$-invertible with the ann-$(b,c)$-inverse $x$ for some positive integer $n$.
	\begin{enumerate}[i)]
		\item If $x(a_1a_2)^n=(a_1a_2)^nx$, then $a_2xa_1(a_2a_1)^n=(a_2a_1)^na_2xa_1$.
		\item If $x\in\operatorname{comm}^2\{a_1a_2\}$, then $a_2xa_1\in\operatorname{comm}^2\{a_2a_1\}$.
	\end{enumerate}
\end{prop}
\begin{proof}
	Part i) is obvious. Consider Part ii). For any $h\in R$ such that $h(a_2a_1)=(a_2a_1)h$, we first show $a_1ha_2(a_1a_2)=(a_1a_2)a_1ha_2$. Indeed,
	\[
	a_1ha_2(a_1a_2)=a_1h(a_2a_1)a_2=a_1(a_2a_1)ha_2=(a_1a_2)a_1ha_2.
	\]
	Thus $x(a_1ha_2)=(a_1ha_2)x$. Also, note that $x\in\operatorname{comm}^2\{a_1a_2\}$ implies $x(a_1a_2)^k=(a_1a_2)^kx$ for all positive integer k. We then conclude that
	\begin{align*}
	(a_2xa_1)h&=a_2x^2(a_1a_2)^{n+1}a_1h\\
	&=a_2x^2a_1(a_2a_1)^{n+1}h\\
	&=a_2x^2a_1h(a_2a_1)^{n+1}\\
	&=a_2x^2(a_1ha_2)(a_1a_2)^na_1\\
	&=a_2(a_1ha_2)x^2(a_1a_2)^na_1\\
	&=ha_2(a_1a_2)x^2(a_1a_2)^na_1\\
	&=h(a_2xa_1).\qedhere
	\end{align*}
\end{proof}

\begin{cor}\label{cor_cline_ann}
	Let $a_1,a_2,b,c\in R$, and suppose that $a_1a_2$ is ann-$(b,c)$-invertible with the ann-$(b,c)$-inverse $x$. If $x(a_1a_2)=(a_1a_2)x$, then $a_2a_1$ is ann-$(a_2b,ca_1)$-invertible with the ann-$(a_2b,ca_1)$-inverse $a_2x^2a_1$.
	
	In particular, if $(a_1a_2)b\in b\mu R$ and $c(a_1a_2)\in R\nu c$, where $\mu,\nu\in R^1$, then $a_2a_1$ is ann-$(a_2b\mu,\nu ca_1)$-invertible with the ann-$(a_2b\mu,\nu ca_1)$-inverse $a_2x^2a_1$.
\end{cor}
\begin{proof}
	Since $x(a_1a_2)=(a_1a_2)x$, Theorem \ref{thm_suff_reverse_ann} shows that $(a_1a_2)^2$ is also ann-$(b,c)$-invertible and its ann-$(b,c)$-inverse is $x^2$. The first part follows from Theorem \ref{thm_cline_ann} above.
	
	We now focus on the second part. Indeed, we only need to verify that $\prescript{\circ}{}{(a_2b\mu)}\subseteq\prescript{\circ}{}{(a_2x^2a_1)}$ and $(\nu ca_1)^\circ\subseteq (a_2x^2a_1)^\circ$. By Corollary \ref{cor_comm_ann}, $(a_1a_2)b\in b\mu R$ and $c(a_1a_2)\in R\nu c$ imply $x(a_1a_2)=(a_1a_2)x$, and, therefore, the other conditions in \eqref{eq_def_ann} are straightforward. For any $t\in\prescript{\circ}{}{(a_2b\mu)}$, we have $ta_2b\mu=0$. Combining with $(a_1a_2)b\in b\mu R$, we get $ta_2(a_1a_2b)=0$. This yields $ta_2a_1a_2x=0$. Note that $a_1a_2x^2=x(a_1a_2)x=x$. Hence $ta_2x=0$, which shows $t\in\prescript{\circ}{}{(a_2x)}\subseteq\prescript{\circ}{}{(a_2x^2a_1)}$. This complete the verification of $\prescript{\circ}{}{(a_2b\mu)}\subseteq\prescript{\circ}{}{(a_2x^2a_1)}$. Similarly, we get $(\nu ca_1)^\circ\subseteq (a_2x^2a_1)^\circ$.
\end{proof}

\begin{thm}\label{thm_cline}
	Let $a_1,a_2,b,c\in S$, and suppose that $a_1a_2$ is $(b,c)$-invertible with the $(b,c)$-inverse $x$. If $(a_1a_2)b\in b\mu S$ and $c(a_1a_2)\in S\nu c$, where $\mu,\nu\in S^1$, then $a_2a_1$ is $(a_2b\mu,\nu ca_1)$-invertible with the $(a_2b\mu,\nu ca_1)$-inverse $a_2x^2a_1$.
	
	In particular, if $x(a_1a_2)=(a_1a_2)x$ and $b=c=d\in S$, then $a_2a_1$ is $(a_2da_1,a_2da_1)$-invertible with the $(a_2da_1,a_2da_1)$-inverse $a_2x^2a_1$.
\end{thm}
\begin{proof}
	By Theorem \ref{thm_comm}, $(a_1a_2)b\in b\mu S$ and $c(a_1a_2)\in S\nu c$ give $x(a_1a_2)=(a_1a_2)x$. Then,
	\begin{align*}
	&\left(a_2x^2a_1\right)\left(a_2a_1\right)\left(a_2b\mu\right)=a_2x^2\left(a_1a_2\right)^2b\mu=a_2\left(xa_1a_2\right)^2b\mu=a_2b\mu,\\
	&\left(\nu ca_1\right)\left(a_2a_1\right)\left(a_2x^2a_1\right)=\nu c\left(a_1a_2\right)^2x^2a_1=\nu c\left(a_1a_2x\right)^2a_1=\nu ca_1.
	\end{align*}
	Assume that $(a_1a_2)b=b\mu e$ and $c(a_1a_2)=f\nu c$, where $e,f\in S$. Also, we can assume that $x^2=bp=qc$, where $p,q\in S$, since $x^2\in bS\cap Sc$ (Theorem \ref{thm_suff_reverse} shows $x^2$ is the $(b,c)$-inverse of $(a_1a_2)^2$). We then claim that $a_2x^2a_1\in a_2b\mu S$. Indeed,
	\[
	a_2b\mu\left(epxa_1\right)=a_2\left(b\mu e\right)pxa_1=a_2\left(a_1a_2b\right)pxa_1=a_2a_1a_2\left(bp\right)xa_1=a_2\left(a_1a_2x^2\right)xa_1=a_2x^2a_1.
	\]
	Similarly, $a_2x^2a_1=(a_2xqf)\nu ca_1\in S\nu ca_1$.
	
	To prove the second part, we will need to obtain $(a_1a_2)d\in da_1S$ and $d(a_1a_2)\in Sa_2d$. From Theorem \ref{thm_comm}, $x(a_1a_2)=(a_1a_2)x$ implies $(a_1a_2)d\in dS$ and $d(a_1a_2)\in Sd$. On the other hand, $da_1a_2x=d$ gives $da_1S=dS$. Thus $(a_1a_2)d\in dS=da_1S$. Similarly, $d(a_1a_2)\in Sa_2d$. This complete the proof.
\end{proof}

Recall Lemma \ref{lem_relation}. We see that if $(a_1a_2)^{((a_1a_2)^k,(a_1a_2)^k)}=x$ (note that $x$ in this case commutes with $a_1a_2$), where $k=\operatorname{ind}(a_1a_2)$, then we have $(a_2a_1)^{((a_2a_1)^{k+1},(a_2a_1)^{k+1})}=a_2x^2a_1$ and $\operatorname{ind}(a_2a_1)\leqslant k+1$. This is known as Cline's formula for Drazin inverses. Conversely, $k\leqslant\operatorname{ind}(a_2a_1)+1$. Therefore, $|\operatorname{ind}(a_2a_1)-k|\leqslant 1$.

\begin{rmk}
	Recently, there are several extensions of Cline's formula for generalized inverses have been introduced (see e.g. \cite{zeng_new_2017}). We also consider the case when $a_1a_2a_1=a_1a_3a_1$. Let $a_1,a_2,a_3,d\in S$ so that $a_1a_2a_1=a_1a_3a_1$ and $a_1a_2$ is $(d,d)$-invertible with the $(d,d)$-inverse $x\in S$. When $x(a_1a_2)=(a_1a_2)x$, we can show that $a_3a_1$ is $(a_3da_1,a_3da_1)$-invertible with $(a_3a_1)^{(a_3da_1,a_3da_1)}=a_3x^2a_1$ which also satisfies $a_3x^2a_1(a_3a_1)=(a_3a_1)a_3x^2a_1$. In brief,
	\[
	a_3x^2a_1(a_3a_1)=a_3x^2a_1a_2a_1=a_3a_1a_2x^2a_1=a_3a_1a_2a_1a_2x^3a_1=a_3a_1a_3a_1a_2x^3a_1=(a_3a_1)a_3x^2a_1.
	\]
	We can also apply this technique to verify (\ref{eq_def_useful}.1) and (\ref{eq_def_useful}.2). In addition, the verification of (\ref{eq_def_useful}.3) and (\ref{eq_def_useful}.4) is nearly identical to that of Theorem \ref{thm_cline}. One thing we need to point out is that we shall show $d(a_1a_2)\in Sa_3d$ before proving $a_3x^2a_1\in Sa_3da_1$. Indeed, from 
	\[
	xa_1a_3d=xa_1a_3(xa_1a_2d)=xa_1a_3a_1a_2xd=xa_1a_2a_1a_2xd=xa_1a_2xa_1a_2d=d,
	\]
	we get $Sd=Sa_3d$. ($x(a_1a_2)=(a_1a_2)x$ implies $d(a_1a_2)\in Sd$.) Thus $d(a_1a_2)\in Sd=Sa_3d$.
\end{rmk}

\section*{Acknowledgments}

This research was supported by the grants from the National Natural Science Foundation of China (No. 11971294).

\bibliography{mybibfile}

\end{document}